\documentclass[article,12pt]{amsart}
\usepackage{amscd}
\usepackage{bbm}
\usepackage{dsfont}
\usepackage{enumerate}
\usepackage{latexsym}

\usepackage{srcltx}
\usepackage{amsfonts}
\usepackage{amssymb}
\usepackage{datetime}
\usepackage{setspace}

\newtheorem{theorem}{Theorem}[section]
\newtheorem{proposition}[theorem]{Proposition}
\newtheorem{lemma}[theorem]{Lemma}
\newtheorem{corollary}[theorem]{Corollary}

\theoremstyle{definition}
\newtheorem{example}[theorem]{Example}
\newtheorem{definition}[theorem]{Definition}

\newtheorem{remark} [theorem] {Remark}
\newtheorem{conjecture} [theorem] {Conjecture}
\begin{document}

\title{ Spectrum of Weighted Composition Operators \\
Part II \\
Weighted composition operators
on subspaces of Banach lattices}

\author{A. K. Kitover}

\address{Community College of Philadelphia, 1700 Spring Garden St., Philadelphia, PA, USA}

\email{akitover@ccp.edu}


\date{\today}


\maketitle

\markboth{A.K.Kitover}{Spectrum of weighted composition operators. II}

\bigskip

\begin{abstract} We describe the spectrum of weighted $d$-isomorphisms of Banach lattices restricted on closed subspaces that are ``rich'' enough to preserve some ``memory'' of the order structure of the original lattice. The examples include (but are not limited to) weighted isometries of Hardy spaces on the polydisk and unit ball in $\mathds{C}^n$.
\end{abstract}

This paper is a continuation of the study of spectrum of weighted composition operators attempted in~\cite[Part 3]{AAK} and~\cite{Ki}. The results of the current paper heavily depend on those in~\cite{AAK} and~\cite{Ki}, and the reader is referred to~\cite{Ki} for the notations not explained here. The main goal is to establish a connection between the spectrum of a weighted $d$-isomorphism $T$ on a Banach lattice $X$ and the spectrum of its restriction on a closed $T$-invariant subspace $Y$ of $X$. Surely, for such a connection to exist and be meaningful we must assume some ``richness'' of the subspace $Y$. The reader is referred to Definitions~\ref{d1} and~\ref{d3} for details but the Hardy spaces $H^p$ considered as subspaces of the corresponding $L^p$ spaces on the unit circle, or a unital uniform algebra as a subspace of the space of all continuous functions on its Shilov boundary represent typical examples of ``rich'' subspaces.

\bigskip

The following notations will be used throughout the paper.

\noindent $X$ - a Banach lattice over the field of complex numbers $\mathds{C}$.

\noindent $Z(X)$ - the center of the Banach lattice $X$.

\noindent $\sigma(T,Y)$ - the spectrum of a bounded linear operator $T$ on a Banach space $Y$.

\noindent $\sigma_r(T,Y) = \{\lambda \in \sigma(T,Y): \; \exists c>0 \; \|\lambda x - Tx \| \geq c \|x\|, \; x \in X \}$.

\noindent $\sigma_{ap}(T,Y) = \sigma(T,Y) \setminus \sigma_r(T,Y)$.

\noindent $\rho(T,Y)$ - spectral radius of $T$ on $Y$.

\noindent \textit{Disjointness preserving operator} - a linear operator $T$ on a Banach lattice $X$ such that
$|x_1| \wedge |x_2| = 0 \Rightarrow |Tx_1| \wedge |Tx_2| =0, \; x_1, x_2 \in X$.

\noindent $d$-\textit{isomorphism} - a linear, bounded, invertible, and disjointness preserving operator on $X$. It is well known (see e.g~\cite{AK}) that the inverse of such an operator also preserves disjointness. \footnote{Often another definition of a $d$-isomorphism is used when instead of invertibility only the equivalence $|x_1| \wedge |x_2| = 0 \Leftrightarrow |Tx_1| \wedge |Tx_2| =0$ is required.}

\bigskip

The following definition was introduced by Glicksberg in~\cite[Definition 4.8.1]{Gl}

\begin{definition} \label{d0} Let $A$ be a closed subspace of $C(K)$. We say that $A$ is \textit{approximating in modulus} if for any nonnegative $f \in C(K)$ and for any positive $\varepsilon$ there is a $g \in A$ such that $\|f - |g| \|_{C(K)} \leq \varepsilon$.

\end{definition}

For our purposes it will be useful to extend the previous definition on the class of Banach lattices.

\begin{definition} \label{d1} Let $X$ be a Banach lattice and $Y$ be a closed subspace of $X$. We will say that $Y$ is an approximating in modulus, or $AIM$- subspace of $X$ if for any $x \in X$ and any positive $\varepsilon$ there is a $y \in Y$ such that $\||x| - |y|\| \leq \varepsilon$.

\end{definition}

\begin{theorem} \label{t1} Let $X$ be a Dedekind complete Banach lattice, $S$ be a $d$-isomorphism of $X$, and $A \in Z(X)$. Let $T = AS$. Let $Y$ be an $AIM$-subspace of $X$ such that $TY \subseteq Y$. Then $\sigma_{ap}(T,X) = \sigma_{ap}(T,Y)$.

\end{theorem}

\begin{proof} We start with recalling a few pretty obvious things. First notice that the $d$-isomorphism $S$ generates an isomorphism of the center $Z(X)$ of $X$ according to the formula $f \rightarrow SfS^{-1}$ (see e.g~\cite[Proposition 8.3]{AK}, \cite{Ab}, and~\cite{AVK}). To this isomorphism corresponds a homeomorphism $\varphi$ of the Stonean compact $K$ of $X$ (we identify $Z(X)$ and $C(K)$) and $supp \, Sx = \varphi^{-1}(supp x), x \in X$ whence
 $$supp \;(Tx) \subseteq \varphi^{-1}(supp \; x) \eqno{(1)}$$
 For $f \in Z(X)$ we will identify the central operator $SfS^{-1}$ and the function $f \circ \varphi$ in $C(K)$.
We will need the formula
$$ T^k = A_k S^k, k \in \mathds{N}$$
where $A_k = A(A \circ \varphi) \ldots (A \circ \varphi^{k-1})$. Recall also that the operator $T$ is regular and that $\|Tx\| = \||T||x|\|, x \in X$. Let us proceed with the proof of the theorem.

The inclusion $\sigma_{ap}(T,Y) \subseteq \sigma_{ap}(T,X)$ is obvious. Let $\lambda \in \sigma_{ap}(T,X)$.
 Let us first consider the case when $\lambda = 0$. Let $x_n \in X, \|x_n\| = 1, Tx_n \rightarrow 0$. Then $\||T||x_n|\| = \|Tx_n\| \rightarrow 0$. Let $y_n \in Y$ be such that $|y_n| - |x_n| \rightarrow 0$. Then $\|Ty_n\| = \||T||y_n|\| \rightarrow 0$ whence $0 \in \sigma_{ap}(T,Y)$.
 Assume now that $\lambda \neq 0$. Without loss of generality we can assume that $\lambda = 1$. Let $x_n \in X$, $\|x_n\|=1$, and $Tx_n - x_n \to 0$.
 For any natural $m$ let $G_m$ be the subset of $K$ defined as
$$G_m = \{k \in K : \; \varphi^m(k) = k, \; \varphi^j(k) \neq k, j \in \mathds{N}, j < m.\}$$
By Frolik's theorem~\cite[Theorem 6.25, p.150]{Wa} $G_m$ is a clopen (maybe empty) subset of $K$. We will denote the corresponding band projection in $X$ by $P_m$. Let also $F_m = \bigcup \limits_{j=1}^m G_m$. It follows from the definition of $G_m$ and from $(1)$ that $P_m T = T P_m$ and that the operator $T^m P_m = P_m T^m P_m$ is a band preserving and therefore a central operator on the band $P_mX$. We have to consider two possibilities.

\noindent $(1)$ There is an $m \in \mathds{N}$ such that $\| P_m x_n \| \mathop \nrightarrow \limits_{n \to \infty} 0$. Then $1 \in \sigma_{ap}(T^m P_m)$. Indeed, $T^m P_m x_n - P_m x_n = P_m T^m x_n - P_m x_n \mathop \rightarrow \limits_{n \to \infty} 0$. The operator $T^m P_m  $ on $P_mX$ can be represented as an operator of multiplication on $M \in C(G_m)$ and therefore there is a point $g \in G_m$ such that $M(g) = 1$. Let us fix a positive $\varepsilon$ and let $\delta = \Big{(}(2 + \|T\|) \big{(} 1 + \sum \limits_{i=1}^{m-1} \|T^i\| \big{)} \Big{)}^{-1}$. Let  $V$ be a clopen neighborhood of $g$ in $G_m$ such that $|M - 1| < \delta$ on $V$ and the sets
$\varphi^{-j}(V)$, $j = 0,1, \ldots , m-1$ are pairwise disjoint. Because $Y$ is an $AIM$-subspace of $X$ we can find a $y \in Y$ such that $\|y\| = 1$ and $\|y - P_V y \| < \delta$ where $P_V$ is the band projection corresponding to the clopen set $V$. Indeed, let $x$ be a positive element from $P_V X$ such that $\|x \| = 1$. Let $\gamma$ be such that $\frac{2\gamma}{1 - \gamma} = \delta$. Let $z \in Y$ and $\||z| - x \| < \gamma$. Then $\|P_V |z| - x \| = \|P_V |z| - P_V x \| < \gamma$ whence $\||z| - P_V |z| \| < 2 \gamma$. But $P_V$ is a band projection and therefore $|z - P_Vz| = |z| - |P_Vz|$ whence $\|z - P_V z\| < 2 \gamma$. Take $y = \frac{z}{\|z\|}$. Then $\|y - P_V y \| < \frac{2 \gamma}{\|z\|} \leq \frac{2\gamma}{1 - \gamma} = \delta$.

Let $v = \sum \limits_{j=0}^{m-1} T^jy$ and $w = P_V \sum \limits_{j=0}^{m-1} T^jy =  \sum \limits_{j=0}^{m-1} T^jP_V y$. Then it is easy to see that
\begin{enumerate}
  \item $\|v - w\| < \delta \big{(} 1 + \sum \limits_{i=1}^{m-1} \|T^i\| \big{)}$.
  \item $\|w\| \geq \|P_Vy\| \geq 1 - \delta$.
  \item $\|Tw - w\| = \|T^m P_V y - y\| = \|(M-1)P_V y\| < \delta$.
\end{enumerate}
It follows immediately from inequalities $(1) - (3)$ that
$$\|v\| > 1 - \delta \big{(} 2 + \sum \limits_{i=1}^{m-1} \|T^i\| \big{)} > 1 - \varepsilon$$
 and that
$$\|Tv - v\| \leq \|v - w\| + \|Tv - Tw\| + \|Tw - w\| < \delta (2 + \|T\|) \big{(} 1 + \sum \limits_{i=1}^{m-1} \|T^i\| \big{)} = \varepsilon.$$
Because $\varepsilon$ can be made arbitrary small we see that $1 \in \sigma_{ap}(T,Y)$.

\noindent $(2)$ For any natural $m$ we have $\|P_m x_n\| \mathop \to \limits_{n \to \infty} 0$. Let us fix a large natural $N$ and let $K_N = K \setminus F_{2N + 1}$. Let $Q$ be the band projection in $X$ corresponding to the clopen set $K_N$ and let $z_n = Qx_n$. Then $\|z_n\| \mathop \to \limits_{n \to \infty} 1$ and $Tz_n - z_n \mathop \to \limits_{n \to \infty} 0$. It follows from the fact that $K$ is extremally disconnected, from the definition of the set $K_N$, and from the Zorn's lemma that there is a maximal by inclusion clopen subset $E$ of $K_N$ such that
 $$\varphi^{-i}(E) \cap \varphi^{-j}(E) = \emptyset , \; 0 \leq i < j \leq 2N. \eqno{(2)}$$
Let $P_E$ be the band projection corresponding to the clopen set $E$. We claim that
$$\liminf \limits_{n \to \infty} \|P_E z_n \| > 0. \eqno{(3)}$$
Indeed, assume to the contrary that $\liminf \limits_{n \to \infty} \|P_E z_n \| = 0$. By choosing if necessary a subsequence we can assume that $\lim \limits_{n \to \infty} \|P_E z_n \| = 0$. Let $E_i = E \cap \varphi^i(E), i \in [2N+2 : 4N+1]$. It follows from the fact that $E$ is a maximal by inclusion clopen subset of $K_N$ satisfying $(2)$ that the sets $\varphi^{-j} (E_i), \; 2N + 2 \leq i \leq 4N+1, \; 0 \leq j \leq i - 1$ make a partition of $K_N$ (the proof is basically the same as the one of the Halmos - Rokhlin lemma). Let $P_{ij}$ be the band projection corresponding to the set $\varphi^{-j}(E_i)$. Then
$P_{ij} T^i z_n = (P_{i0} \circ \varphi^i)T^i z_n = T^i P_{i0}z_n$. But $P_{i0}$ is the band projection corresponding to the set $E_i$ whence
$\|P_{i0}z_n\| \leq \|P_E z_n\| \mathop \rightarrow \limits_{n \to \infty} 0$ and therefore $P_{ij} T^i z_n  \mathop \rightarrow \limits_{n \to \infty} 0$. But $T^i z_n - z_n \mathop \rightarrow \limits_{n \to \infty} 0$ whence $P_{ij}z_n \mathop \rightarrow \limits_{n \to \infty} 0$. Recalling that
$\sum \limits_{i,j} P_{ij} = Q$ and that $Qz_n = z_n$ we see that $z_n \mathop \rightarrow \limits_{n \to \infty} 0$, a contradiction.

Let $\mathcal{E} = \{E : E $ is maximal by inclusion closed subset of $ K_N$ satisfying $(2) \}$. Notice that because $\varphi$ is a homeomorphism we have $E \in \mathcal{E} \Leftrightarrow \varphi(E) \in \mathcal{E}$.
Let $s_n = \sup \limits_{E \in \mathcal{E}} \|P_E z_n\|$. It follows from $(3)$ that $\liminf \limits_{n \to \infty} s_n = s$ where
$0 < s \leq 1$. Let $E_n \in \mathcal{E}$ be such that $\|(P_n \circ \varphi^N)  z_n\| \geq s_n - 1/n$ where $P_n = P_{E_n}$. Then $\liminf \|(P_n \circ \varphi^N) z_n \| = s$ and we can assume without loss of generality that $\lim \|(P_n \circ \varphi^N) z_n \| = s$.

For any $n \in \mathds{N}$ choose $y_n \in Y$ such that $\||y_n| - |P_n z_n| \| < 1/n$. Notice that then
$$\||P_n y_n| - |P_nz_n| \| =\|P_n (|y_n| - |P_n z_n|)\| \leq 1/n. \eqno{(3^\prime)}$$
 Let
 $$ g_n = \sum \limits_{j= 0}^{2N} \big{(} 1 - \frac{1}{\sqrt{N}} \big{)}^{|j - N|} T^j y_n . $$
 We claim that
 $$\liminf \limits_{n \to \infty} \|g_n\| \geq s \eqno{(4)}$$.
 Indeed
  $$\|g_n\| \geq \|(P_n \circ \varphi^N) g_n \| = \| \sum \limits_{j= 0}^{2N} \big{(} 1 - \frac{1}{\sqrt{N}} \big{)}^{|j - N|} (P_n \circ \varphi^N) T^j y_n \|.$$
 Consider $i = N$ and the corresponding term in the sum above which is $(P_n \circ \varphi^N) T^N y_n $. Notice that
 $$(P_n \circ \varphi^N) T^N y_n = (P_n \circ \varphi^N) A_N S^N y_n = A_N (P_n \circ \varphi^N)S^N y_n =$$
 $$ =A_N S^N S^{-N} (P_n \circ \varphi^N)S^N y_n = T^NP_n y_n.$$
 Next notice that
 $$ \lim \limits_{n \to \infty} \|T^N P_n y_n \| = \lim \limits_{n \to \infty} \|T^N P_n z_n \| = $$
$$= \lim \limits_{n \to \infty} \|(P_n \circ \varphi^N)T^Nz_n \| = \lim \limits_{n \to \infty} \|(P_n \circ \varphi^N)z_n\| = s.$$
Now consider $j \in [0 : 2N], j \neq N$. It follows from $(2)$ and $(3^\prime)$ that
 $$ \lim \limits_{n \to \infty} \|(P_n \circ \varphi^N)T^j y_n\| =\lim \limits_{n \to \infty} \|(P_n \circ \varphi^N)T^j P_n y_n\| = $$
 $$ = \lim \limits_{n \to \infty} \|(P_n \circ \varphi^N)(P_n \circ \varphi^j) T^j y_n\| =0,$$
 and $(4)$ is proved.

On the other hand
$$ Tg_n -g_n =- \big{(}1 - \frac{1}{\sqrt{N}}\big{)}^N y_n - \frac{1}{\sqrt{N}} \sum \limits_{j=1}^N \big{(}1 - \frac{1}{\sqrt{N}}\big{)}^{|j-N|} T^j y_n  $$
$$ + \frac{1}{\sqrt{N}} \sum \limits_{j=N+1}^{2N} \big{(}1 - \frac{1}{\sqrt{N}}\big{)}^{j-N} T^j y_n + \big{(}1 - \frac{1}{\sqrt{N}}\big{)}^N T^{2N+1}y_n \eqno{(5)}$$
Let $p_n = \sum \limits_{j=1}^N \big{(}1 - \frac{1}{\sqrt{N}}\big{)}^{|j-N|} T^j y_n $ and $q_n = \sum \limits_{j=N+1}^{2N} \big{(}1 - \frac{1}{\sqrt{N}}\big{)}^{N-j} T^j y_n$. The sequences $p_n$ and $q_n$ are norm bounded and we claim that
$$\|p_n + q_n\| - \|p_n - q_n\| \mathop \rightarrow \limits_{n \to \infty} 0 \eqno{(6)}$$.
To prove $(6)$ it is enough to prove that $|p_n| \wedge |q_n| \mathop \rightarrow \limits_{n \to \infty} 0$. But $|p_n| \leq \sum \limits_{j=1}^N \big{(}1 - \frac{1}{\sqrt{N}}\big{)}^{|j-N|} |T|^j |y_n| $, $|q_n| \leq \sum \limits_{j=N+1}^{2N} \big{(}1 - \frac{1}{\sqrt{N}}\big{)}^{|j-N|} |T|^j |y_n| $, $\sum \limits_{j=1}^N \big{(}1 - \frac{1}{\sqrt{N}}\big{)}^{|j-N|} |T|^j |y_n|  - \sum \limits_{j=1}^N \big{(}1 - \frac{1}{\sqrt{N}}\big{)}^{|j-N|} |T|^j |P_nz_n|  \mathop \rightarrow \limits_{n \to \infty} 0$, $\sum \limits_{j=N+1}^{2N} \big{(}1 - \frac{1}{\sqrt{N}}\big{)}^{|j-N|} |T|^j |y_n|  - \sum \limits_{j=N+1}^{2N} \big{(}1 - \frac{1}{\sqrt{N}}\big{)}^{|j-N|} |T|^j |P_nz_n | \mathop \rightarrow \limits_{n \to \infty} 0$, and in view of $(2)$
$$ \sum \limits_{j=1}^N \big{(}1 - \frac{1}{\sqrt{N}}\big{)}^{|j-N|} |T|^j |P_nz_n| \wedge \sum \limits_{j=N+1}^{2N} \big{(}1 - \frac{1}{\sqrt{N}}\big{)}^{|j-N|} |T|^j |P_nz_n | = 0.$$
and $(6)$ is proved.

 Let $h_n = - \big{(}1 - \frac{1}{\sqrt{N}}\big{)}^N y_n + \frac{1}{\sqrt{N}} \sum \limits_{j=1}^{2N} \big{(}1 - \frac{1}{\sqrt{N}}\big{)}^{|j|} T^j y_n + \big{(}1 - \frac{1}{\sqrt{N}}\big{)}^N T^{2N+1}y_n $. It follows from $(6)$ that $\|Tg_n - g_n \| - \|h_n\| \to 0$. Notice that
$h_n = \frac{1}{\sqrt{N}}g_n - \big{(}1 - \frac{1}{\sqrt{N}}\big{)}^{N+1}y_n + \big{(}1 - \frac{1}{\sqrt{N}}\big{)}^N T^{2N+1}y_n $, that $\limsup \|y_n\| = \limsup \|P_nz_n\| =s$, and that
$ \limsup \|T^{2N+1}y_n\| = \limsup \|T^{2N+1}P_n z_n\| = \limsup \|(P_n \circ \varphi^{2N+1})T^{2N+1} z_n\| = \limsup \|(P_n \circ \varphi^{2N+1}) z_n\| \leq s$ (we used here that $T^{2N+1}z_n - z_n \rightarrow 0$ and that $\varphi^{-(2N+1)}(E) \in \mathcal{E}$).  Finally we get that
$$ \limsup \|Tg_n - g_n \| \leq \big{[}\frac{1}{\sqrt{N}} +2\big{(}1 - \frac{1}{\sqrt{N}}\big{)}^N \big{]} \liminf \|g_n\|.$$
Because $N$ can be chosen arbitrary large we see that $1 \in \sigma_{ap} (T,Y)$.

\end{proof}

\begin{corollary} \label{c1} Assume conditions of Theorem~\ref{t1} and assume additionally that the powers $T^n, n=0, 1, \ldots$ are pairwise disjoint in the lattice $L_r(X)$ of all regular operators on $X$. Then $\sigma(T,Y)$ is rotation invariant.
\end{corollary}

\begin{proof} It follows from the conditions of the corollary and Theorem 3.14 in~\cite{AAK} that the set $\sigma_{ap}(T,X)$ is rotation invariant whence by Theorem~\ref{t1} $\sigma_{ap} (T,Y)$ is rotation invariant. Because the set $\sigma(T,Y) \setminus \sigma_{ap}(T,Y)$ is open in $\mathds{C}$ we see that $\sigma(T,Y)$ is rotation invariant.

\end{proof}

The next lemma will be used in the sequel but it also might be of independent interest because it provides some information about the residual spectrum
of weighted $d$-isomorphisms.

\begin{lemma} \label{l2} Let $X$ be a Dedekind complete Banach lattice, $S$ be a $d$-isomorphism of $X$, $A \in Z(X)$, $supp \, A = K$, and $T = AS$. Assume that $\lambda \in \sigma(T)\setminus \{0\}$ and $(\lambda I - T)X = X$ (i.e $\lambda \in \sigma_r(T^\star, X^\star))$. Then there is a sequence $\{u_i\}_{i \in \mathds{Z}}$ of nonzero elements in $X$ such that
\begin{itemize}
  \item The elements $u_i, \; i \in \mathds{Z}$ are pairwise disjoint in $X$.
  \item $\sum \limits_{i= -\infty}^\infty \|u_i\| < \infty$.
  \item $Tu_i = \lambda u_{i+1}, \; i \in \mathds{Z}$.
\end{itemize}
\end{lemma}

\begin{proof} We can assume without loss of generality that $\lambda = 1$. Let $x \in X \setminus \{0\}$ be such that $Tx = x$. Let $E = supp \, x$.
The condition $supp \, A = K$ guarantees that $\varphi(E) = E$. Let $X_E$ be the band in $X$ corresponding to $E$. Then the band projection $P_E$ commutes with $T$ and therefore $( I - T)X_E = X_E$ whence $1 \in \sigma_r(T^\star, X_E^\star)$. Because the set $\sigma_r(T^\star, X_E^\star)$ is open in $\mathds{C}$ we can find $\varepsilon > 0$ and $y \in X_E \setminus \{0\}$ such that $Ty = (1 + \varepsilon) y$. Let $F = supp \, y$ then $F \subseteq E$ and $\varphi(F) = F$. Applying this procedure once again (and decreasing $\varepsilon$, if necessary) we obtain a nonempty clopen subset $H$ of $K$ and elements $u,v,w \in X_H$ such that
\begin{itemize}
  \item $\varphi(H) = H$.
  \item $supp \, u = supp \, v = supp \, w = H$.
  \item $Tu = u$, $Tv = (1 + \varepsilon)v$, and $Tw = (1 - \varepsilon)w$.
\end{itemize}
Notice that the set $H$ cannot contain $\varphi$-periodic points. Indeed, otherwise using the Frolik's theorem and the fact that $Tu = u$ we would be able to find a nontrivial band $B$ in $X$ and $m \in \mathds{N}$ such that $T^m | B = I |B$ whence $(T^\star)^m | B^\star = I | B^\star$ in contradiction with our assumption that $1 \in \sigma_r(T^\star)$.

Multiplying, if necessary, the elements $u$, $v$, and $w$ by appropriate positive constants we can find a clopen subset $L$ of $H$ such that $P_L|u| \leq P_L |v| \leq 2 P_L |u|$ and $P_L|u| \leq P_L |w|$. We claim that there is an $N \in \mathds{N}$ such that $\forall n \geq N \; L \cap \varphi^{-n}(L) = \emptyset$. To prove it notice that $ |T^n P_L u| = |A_n S^n P_L u| = |A_n P_L S^n P_L S^{-n} S^n u| = |A_n (P_L \circ \varphi^n)S^n u| = |P_{\varphi^{-n}(L)}T^n u| = |P_{\varphi^{-n}(L)} u|$. Similarly, $|T^n P_L v |= (1 + \varepsilon)^n |P_{\varphi^{-n}(L)} v|$. Let $N \in \mathds{N}$ be such that $(1 +\varepsilon)^N \geq 3$. Assume contrary to our claim that for some $n \geq N$ the intersection
$M = L \cap \varphi^{-n}(L) \neq \emptyset$. Recalling that $|T^n x| = |T||x|, x \in X, n \in \mathds{N}$ we see that
$$ |P_M 2u| = |P_M T^n P_L 2u| \geq |P_M T^n P_L v| = $$
$$=(1 + \varepsilon)^n)|P_M v| \geq 3|P_M v| \geq |P_M 3u| $$
whence $P_M u = 0$ in contradiction with $ M \subset supp \, u$. Next, because $L$ does not contain $\varphi$ periodic points we can find a clopen nonempty subset $R$ of $L$ such that the sets $\varphi^{-n}(R), n=0, 1, \ldots, N$ are pairwise disjoint. Therefore the sets $\varphi^{-n}(R), n \in \{0\} \cup \mathds{N}$ are pairwise disjoint, and because $\varphi$ is a homeomorphism the same is true for the sets $\varphi^n(R), n \in \mathds{Z}$.
Let $u_i = P_{\varphi^{-i}(R)}u, i \in \mathds{Z}$. Similarly we define the elements $v_i$ and $w_i$.  Then $Tu_i = u_{i+1}$, $Tv_i = (1 + \varepsilon v_{i+1}$, and $Tw_i = (1 - \varepsilon)w_{i+1}$. Therefore $|u_n| = |T|^n |u_0| \leq |T|^n |w_0| = (1 - \varepsilon)^n |w_n|$ whence
$\|u_n\| \leq (1 - \varepsilon)^n \|w\|, n \in \mathds{N}$. Next for any $n \in \mathds{N}$ we have $T^n u_{-n} = u_0$ and $T^n v_{-n} = (1 + \varepsilon)^n v_0$ whence $|T^n u_{-n}| \leq |T^n (1+\varepsilon)^{-n}v_{-n}|$. Because $supp \, A = K$ the last inequality is equivalent to
$|u_{-n}| \leq (1 + \varepsilon)^{-n}|v_{-n}|$ whence $\|u_{-n}\| \leq (1 + \varepsilon)^{-n}\|v\|$.
\end{proof}

\begin{remark} \label{r1} $(1)$ Because $SA = S(AS)S^{-1}$ the conclusion of Lemma~\ref{l2} is valid for operators of the form $SA$.

$(2)$ The proof of Lemma~\ref{l2} shows that the condition $supp \; A = K$ can be substituted by the following one. For any $\mu$ from some open neighborhood of $\lambda$ and for any nonzero $x \in X$ such that $Tx = \mu x$ we have $\varphi (supp \; x) = supp \; x$.
\end{remark}

\begin{theorem} \label{t7} Assume conditions of Theorem~\ref{t1}. Then $\sigma(T,X) \subseteq \sigma(T,Y)$.
\end{theorem}

\begin{proof} Assume to the contrary that there is $\lambda \in \sigma(T,X) \setminus \sigma(T,Y)$. By Theorem~\ref{t1} $\lambda \in \sigma_r(T,X)$ and therefore we can assume without loss of generality that $\lambda = 1$. Then $\sigma_{ap}(T,X) \cap \Gamma \neq \Gamma$ and it follows from~\cite[Proposition 12.15]{AAK} that $X$ is the direct sum of two disjoint $T$-invariant bands $X_1$ and $X_2$ (band $X_1$ might be empty) such that some power of $T$ restricted on $X_1$ is a multiplication operator and $\Gamma \subset \sigma_r(T,X_2)$. Consider the subspace $Y_2$ of $X$
defined as $Y_2 = cl \;\{P_2 y : \; y \in Y \}$. Clearly $TY_2 \subseteq Y_2$. It is easy to see that $Y_2$ is an $AIM$ subspace of $X_2$. Indeed, let $x \in X_2$ and let $y_n \in Y$ be such that $|y_n| \rightarrow |x|$. Then $|P_2y_n| \rightarrow |P_2x| = |x|$. Moreover, $\Gamma \cap \sigma(T,Y_2) = \emptyset$. To see it notice that for any $\gamma \in \Gamma$ the operator $\gamma I -T$ is bounded from below on $Y_2$ because $\gamma \in \sigma_r(T,X_2)$. On the other hand consider $y \in Y$. Because the operator $I - T$ is invertible on $Y$ there is $z \in Y$ such that $z - Tz = y$ whence $P_2 z - TP_2 z = P_2y$. Thus the operator $(I - T)|Y_2$ is bounded from below and its image contains a dense subspace of $Y_2$ whence it is invertible on $Y_2$. It remains to notice that because the resolvent set and the residual spectrum of a bounded operator are open in $\mathds{C}$ the operator $(\gamma I - T)|Y_2$ is invertible for any $\gamma \in \Gamma$.

The previous paragraph shows that without loss of generality we can assume that $\Gamma \subset \sigma_r(T,X)$ and $\Gamma \cap \sigma(T,Y) = \emptyset$. Then $\Gamma \subset \sigma_r(T^{\star \star})$. Assume first that for any $\mu$ in some open neighborhood of $1$ and for any nonzero $u \in X^\star$ such that $T^\star u = \mu u$ we have $\psi(supp \; u) = supp \; u$ where $\psi$ is the homeomorphism of the Stonean space $Q$ of $X^\star$ generated by the isomorphism $f \rightarrow (S^\star)^{-1})fS^\star, \; f \in Z(X^\star)$. Then by Lemma~\ref{l2} and Remark~\ref{r1} there are pairwise disjoint elements $u_i \in X^\star, i \in \mathds{Z}$ such that $u_0 \neq 0$, $T^\star u_i = u_{i+1}, i \in \mathds{Z}$, and $\sum \limits_{i= - \infty}^\infty \|u_i\| < \infty$. Let $y \in Y$, then $|y| = fy$ where $f \in Z(X)$ and $|f| \equiv 1$. Let $F = f^\star \in Z(X^\star)$ and let
$v_i = (S^\star)^i F (S^\star)^{-i} u_i$. it is immediate to see that $\|v_i\| = \|u_i\|$ and that $T^\star v_i = v_{i+1}, i \in \mathds{Z}$. For any
$\gamma \in \Gamma$ let $w_\gamma = \sum \limits_{i= - \infty}^\infty \gamma^{-i} v_i$. Then $T^\star w_\gamma = \gamma w_\gamma$ and therefore (because $\sigma(T,Y) \cap \Gamma = \emptyset)$ we have
$$ \forall \gamma \in \Gamma \; \langle y, w_\gamma \rangle = \sum \limits_{i= - \infty}^\infty  \gamma^{-i} \langle y, v_i \rangle =0. $$
Therefore
$$ \langle |y|, u_0 \rangle = \langle y, v_0 \rangle = \int \limits_0^{2\pi} \big{(} \sum \limits_{i= - \infty}^\infty  e^{-i\theta} \langle y, v_i \rangle \big{)} d\theta = 0.$$
But $Y$ is an $AIM$ subspace of $X$ whence $u_0 =0$, a contradiction.

It remains to consider the case when there are $\lambda \neq 0$ and $u \in X^\star \setminus \{0\}$ such that $\lambda \not \in \sigma(T,Y)$,  $T^\star u = \lambda u$, and
$E \subsetneqq \psi^{-1}(E)$ where $E = supp \; u$. Notice that $s^\star A^\star (S^\star)^{-1} \equiv 0$ on $\psi^{-1}(E) \setminus E$. Without loss of generality we can assume that $\lambda = 1$.
Let $F = E \setminus \psi(E)$. Notice that $\psi^i(F) \cap \psi^j(F) = \emptyset, i \neq j, \; i,j \in \mathds{Z}$. Let $u_n = P_n u, n \in \mathds{N}$, where $P_n$ is the band projection corresponding to the set $\psi^n(F)$. Then $T^\star u_n = u_{n-1}, n \in \mathds{N}$, $u_0 \neq 0$, and $T^\star u_0 = 0$. Let $y \in Y$ and introduce $v_n, \; n=0,1, \ldots$ like in the previous part of the proof. For any $\alpha \in \mathds{C}$ such that $|\alpha| < 1$ the series
 $w(\alpha) = \sum \limits_{n=0}^\infty \alpha^i v_i$ converges in norm and $T^\star w(\alpha) = \alpha w(\alpha)$
 The function $W(\alpha) = \langle y,w(\alpha) \rangle$ is analytic in the open unit disk and identically zero on the intersection of the unit disk and the resolvent set of $T|Y$ whence it is identically zero and $\langle |y|, u_0 \rangle = \langle y, v_0 \rangle =0$. We conclude that $u_0 = 0$, a contradiction.

\end{proof}

\begin{definition} \label{d2} Let $X$ be a Banach lattice and $Y$ be a closed subspace of $X$. We will say that $Y$ is an analytic subspace of $X$ if for any nonzero band $E$ in $X$ the implication holds $y \in Y , y \perp E \Rightarrow y =0$.
\end{definition}

\begin{theorem} \label{t2} Let $Y$ be an analytic $AIM$ subspace of a Banach lattice $X$. Let $S$ be a $d$-isomorphism of $X$, $A \in Z(X)$, $T = AS$, and $TY \subseteq Y$. Then the set $|\sigma(T,Y)| = \{|\lambda| : \lambda \in \sigma(T,Y) \}$ is a connected subset of $\mathds{R}$.
\end{theorem}

\begin{proof} Assume to the contrary that there is a positive $r$ such that $\sigma(T,Y) = \sigma_1 \cup \sigma_2$ where $\sigma_1$ and $\sigma_2$ are nonempty subsets of $\mathds{C}$,
$\sigma_1 \subset \{\lambda \in \mathds{C} : \; |\lambda| < r \}$ and $\sigma_2 \subset \{\lambda \in \mathds{C} : \; |\lambda| > r \}$. Let $Y_1$ and $Y_2$ be the corresponding spectral subspaces of the operator $T|Y$. By Theorem~\ref{t7} $\sigma(T,X) \cap r\Gamma = \emptyset$. Let $X_1$ and $X_2$ be the corresponding spectral subspaces. Then clearly $Y_i \subseteq X_i, \; i=1,2$. It follows from~\cite[Theorem 13.1]{AAK} that \footnote{Strictly speaking we should apply Theorem 13.1 from~\cite{AAK} to the operator $T^{\star \star}$ but it provides the desired result.}  $X_1 \perp X_2$ whence $Y_1 \perp Y_2$ in contradiction with our assumption that $Y$ is an analytic subspace of $X$.
\end{proof}

\begin{corollary} \label{c2} Assume conditions of Corollary~\ref{c1} and assume additionally that $Y$ is an analytic subspace of $X$. Then $\sigma(T,Y)$ is either a circle or an annulus centered at $0$.

\end{corollary}

Now we will consider some examples.

\begin{example} \label{e1} Let $(\Gamma, m)$ be the unit circle with the normalized Lebesgue measure $m$. Let $X$ be a Dedekind complete Banach lattice such that $L^\infty (m) \subseteq X \subseteq L^1(m)$. Let $Y = \{x \in X : \; \int \limits_0^{2\pi} x(e^{i\theta}) e^{in\theta}dm =0, \; n \in \mathds{N} \}$. Let  $T$ be a $d$-isomorphism of $X$ such that $TY \subseteq Y$ and the powers $T^n, \; n=0, 1, \ldots $ are pairwise disjoint. Then $\sigma(T,Y)$ is either a circle or an annulus centered at $0$.
\end{example}

\begin{proof}First notice that $Y$ is a closed subspace of $X$ because $X \subseteq L^1(m)$. To prove that $Y$ is an $AIM$ of $X$ let us consider
$x \in X$ and the sequence $z_n = |x| + \frac{1}{n} \mathbf{1}, n \in \mathds{N}$ where $\mathbf{1}$ is the constant function identically equal to 1. Then,
because $L^\infty(m) \subseteq X$ we have $z_n \in X$ and $z_n \mathop \to \limits_X  |x|$. Notice now that $\ln{z_n} \in L^1(m)$ and therefore (see e.g.~\cite[page 53]{Ho}) there are $h_n \in H^1$ such that $|h_n| = z_n$. Then $h_n \in Y$ and therefoe $Y$ is an $AIM$ subspace of $X$. It remains to notice that $Y$ is an analytic subspace of $X$, because $Y \subseteq H^1$, and apply Corollary~\ref{c2}.

\end{proof}

\begin{example} \label{e2} Let $p \in \mathds{N}$. Let $\Omega$ be either $(\Gamma^p, m)$ - the $p$-dimensional unit torus with the Haar measure $m$ or $(B_p, m)$ the unit ball in $\mathds{C}^p$ with the normalized Lebesgue measure $m$ (see~\cite[1.4.1]{Ru1}).  Let $X$ be a Banach lattice with order continuous norm such that $L^\infty (\Omega) \subseteq X \subseteq L^1(\Omega)$. Let $Y = X \cap H^1(\Omega)$. Let  $T$ be a $d$-isomorphism of $X$ such that $TY \subseteq Y$ and the powers $T^n, \; n=0, 1, \ldots $ are pairwise disjoint. Then $\sigma(T,Y)$ is either a circle or an annulus centered at $0$.
\end{example}

\begin{proof}  Let $x \in X$. Because the norm in $X$ is order continuous $|x|$ can be approximated by norm in $X$ by strictly positive functions $z_n$ from $C(\Omega)$. There are $h_n \in H^\infty(\Omega)$ such that $|h_n| = z_n$. In the case of polydisc it follows from Theorem 3.5.3 in~\cite{Ru} and in the case of the unit ball in $\mathds{C}^p$ from the deep results of Alexandrov~\cite{Al} and L{\o}w~\cite{Lo} (see also~\cite{Ru2}). The rest of the proof goes like in Example~\ref{e1}.

\end{proof}

The condition that a subspace of a Banach lattice is $AIM$ is quite restrictive and in the following part of the paper we will weaken it but at the price of imposing additional conditions either on the space $X$ or on the operator $T$.

Recall that if $X$ is an arbitrary Banach lattice and $T$ is an order continuous operator on $X$ then $T$ can be extended in the unique way on the Dedekind completion $\hat{X}$ of $X$ (see~\cite{Ve}).

\begin{theorem} \label{t8} Let $X$ be a Banach lattice, $S$ be a $d$-isomorphism of $X$, $A \in Z(X)$, and $T = AS$. Let $\hat{X}$ be the Dedekind completion of $X$ and $\hat{T}= \hat{A} \hat{S}$ be (the unique) extension of $T$ on $\hat{X}$. Assume also that $\hat{X}$ has the weak sequential Fatou property.
Then $\sigma_{ap}(T,X) = \sigma_{ap}(\hat{T}, \hat{X})$.
\end{theorem}

\begin{proof} Assume first that $\hat{X}$ has the sequential Fatou property. The inclusion $\sigma_{ap}(T,X) \subseteq \sigma_{ap}(\hat{T}, \hat{X})$ is trivial. If $0 \in \sigma_{ap}(\hat{T}, \hat{X})$ then $\hat{A}$ is not invertible in $Z(\hat{X})$ and it is routine to see that $0 \in \sigma_{ap}(T,X)$. Thus it is enough to prove that if $1 \in \sigma_{ap}(\hat{T}, \hat{X})$ then $1 \in \sigma_{ap}(T,X)$. We will use the notations from the proof of Theorem~\ref{t1}. Consider $x_n \in \hat{X}$ such that $\|x_n\| = 1$ and $Tx_n - x_n \rightarrow 0$. Assume that there is a natural $m$ such that $\|P_m x_n \| \mathop \nrightarrow \limits_{n \to \infty} 0$. Then there is a point $t \in G_m$ such that $\hat{A}_m(t) = 1$. For every $k \in \mathds{N}$ we can a clopen subset $V_k$ of $G_m$ such that the sets $V_k, \varphi^{-1}(V_k), \ldots , \varphi^{-(m-1)}(V_k)$ are pairwise disjoint and $|\hat{A}_m(k) -1| < 1/k, k \in V_k$. Let $x_k \in X$ be such that $\|x_k\| = 1$ and $x_k$ belongs to the band in $X$ corresponding to the set $V_k$. Let $y_k = \sum \limits_{i=0}^{m-1} T^i x_k$. Then it is immediate to see that $\|Ty_k -y_k\| = o(\|y_k\|), k \rightarrow \infty$.

Now assume that $P_m x_n \mathop  \rightarrow \limits_{n \to \infty}  0$ for any $m \in \mathds{N}$. Let $N$, $z_n$, $E_n$, and $P_n$ be as in the proof of Theorem~\ref{t1}. Let $u_n = P_n z_n$ and
 $$ g_n = \sum \limits_{j= 0}^{2N} \big{(} 1 - \frac{1}{\sqrt{N}} \big{)}^{|j - N|} T^j u_n . $$
Because $\hat{X}$ has the Fatou property for every $n \in \mathds{N}$ there is $v_n \in X$ such that $|v_n| \leq |u_n|$ and $\|v_n\| > \|u_n\| - 1/n$.
Let
 $$ h_n = \sum \limits_{j= 0}^{2N} \big{(} 1 - \frac{1}{\sqrt{N}} \big{)}^{|j - N|} T^j v_n . $$
 Then like in the proof of Theorem~\ref{t1} we can see that
 $$ \limsup \|Th_n - h_n \| \leq \big{[}\frac{1}{\sqrt{N}} +2\big{(}1 - \frac{1}{\sqrt{N}}\big{)}^N \big{]} \liminf \|h_n\|.$$

 Assume now that $\hat{X}$ has only the weak sequential Fatou property. Then the formula $|||x||| = \inf \lim \limits_{n \to \infty} \|x_n\|$, where the infimum is taken over all the sequences $\{x_n\}$ of elements of $\hat{X}$ such that $|x_n| \uparrow |x|$, defines an equivalent norm on $\hat{X}$ with the Fatou property. Because changing a norm to an equivalent one does not change the spectrum of a bounded linear operator we are done.
\end{proof}

\begin{theorem} \label{t9} Assume conditions of Theorem~\ref{t8}. Then $\sigma(\hat{T}, \hat{X}) \subseteq \sigma(T,X)$.
\end{theorem}

\begin{proof} Assume to the contrary that there is $\lambda \in \sigma(\hat{T}, \hat{X}) \setminus \sigma(T,X)$. Then by Theorem~\ref{t8} $\lambda \in
\sigma_r(\hat{T}, \hat{X})$. We can assume that $\lambda = 1$. Then (see~\cite[Chapter 13]{AAK}) there are two possibilities.

\noindent $(1)$ $\Gamma \cap \sigma(T,X) = \emptyset$. Then by Proposition 13.3 in~\cite{AAK} $X$ is the direct sum of two disjoint $T$-invariant bands $X_1$ and $X_2$ such that $\sigma(T,X_1) \subset \{ \lambda \in \mathds{C}: \; |\lambda| < 1\}$ and $\sigma(T,X_2) \subset \{ \lambda \in \mathds{C}: \; |\lambda| > 1\}$. Then $\hat{X}$ is the direct sum of disjoint $\hat{T}$-invariant bands $\hat{X_1}$ and $\hat{X_2}$ and $\rho(\hat{T}|\hat{X_1}) =
\rho(T,X_1) < 1$. The restrictions $T|X_2$ and $\hat{T}|\hat{X_2}$ are a $d$-isomorphisms on $X_2$ and respectively $\hat{X_2}$. Let $R : X_2 \to X_2$ be the inverse of $T|X_2$. Then $\hat{R}$ is the inverse of $\hat{T}|\hat{X_2}$ and $\rho(R) = \rho(\hat{R}) < 1$. Therefore $\sigma(\hat{T},\hat{X}) \cap \Gamma = \emptyset$, a contradiction.

\noindent $(2)$ Assume now that $\sigma(\hat{T},\hat{X}) \cap \Gamma \neq \emptyset$. Let $K$ be the Stonean compact of $\hat{X}$ and $\varphi$ be the homeomorphism of $K$ generated by the operator $\hat{S}$. For any $m \in \mathds{N}$ let $F_m$ be the clopen subset of $K$ that consists of all $\varphi$-periodic points of period $\leq m$, and let $\hat{B}_m$ be the band in $\hat{X}$ corresponding to $F_m$. Let also $B_m = X \cap \hat{B}_m$.
Notice that for any $m$ the bands $B_m$ and $\hat{B}_m$ are invariant for $T$ and respectively $\hat{T}$. Let $\dot{T}_m$ and $\dot{\hat{T}}_m$ be the factor operators on $X/B_m$ and $\hat{X}/ \hat{B_m}$, respectively. Then it follows from Proposition 12.15 in~\cite{AAK} that there is such $m \in \mathds{N}$ that $\sigma(\dot{T}_m) \cap \Gamma = \emptyset$ and $\Gamma \subset \sigma_r(\dot{\hat{T}}_m)$. Because $\widehat{X/B_m} = \hat{X}/\hat{B_m}$ we are now in the conditions of part $(1)$ of the proof and come to a contradiction.
\end{proof}

\begin{remark} \label{r2} It is not known to the author whether the condition that $\hat{X}$ has the weak Fatou property in Theorems~\ref{t8} and~\ref{t9} can be dropped.
\end{remark}

In connection with Theorems~\ref{t8} and~\ref{t9} it might be interesting to obtain a criterion for a point in the spectrum of a weighted $d$-isomorphism $T=AS$ to belong either to $\sigma_{ap}(T)$ or to $\sigma_r(T)$. A complete answer at the present is known only under additional condition that $\sigma(S) \subseteq \Gamma$ (see Theorem~\ref{t3} below and the results that follow it). Nevertheless, Lemma~\ref{l2} allows us to obtain the following result.

\begin{theorem} \label{t10} Let $X$ be a Banach lattice with order continuous norm, $S$ be a $d$-isomorphism of $X$, $A \in Z(X)$, and $T = AS$. Let
$\lambda \in \sigma_r(T)$. Then

\noindent $(1)$ There is a nonzero band $E$ in $X$ such that the bands $S^i E, i \in \mathds{Z}$ are pairwise disjoint.

\noindent $(2)$ $X$ is the union of four disjoint $T$-invariant bands $F_1$, $F_2$, $F_3$, and $F_4$ (any of the bands $F_2$, $F_3$, and $F_4$ can be $0$) such that

\noindent $F_1 = \big{\{} \bigcup \limits_{i= - \infty}^\infty S^i E \big{\}}^{dd}$.

\noindent $\sigma(T|F_2) \subset \{\alpha \in \mathds{C}: |\alpha| < |\lambda|\}$.

\noindent $\sigma(T|F_3) \subset \{\alpha \in \mathds{C}: |\alpha| > |\lambda|\}$.

\noindent Some power of the operator $T|F_4$ is a multiplication operator and $\lambda \not \in \sigma(T|F_4)$.

\end{theorem}

\begin{proof} First notice that (see e.g.~\cite[Theorem 2.4.2.]{MN}) that $X$ is Dedekind complete and its canonical image in the second dual $X^{\star \star}$ is an ideal in $X^{\star \star}$.  It follows from Lemma~\ref{l2} and the proof of Theorem~\ref{t7} that there is a nonzero band $H$ in $X^\star$ such that the bands $(S^\star)^i, i \in \mathds{Z}$ are pairwise disjoint. Let $H_1$ be a maximal by inclusion band in $X^\star$ with the same property (it exists because $X^\star$ is Dedekind complete) let $L_1 = \big{\{} \bigcup \limits_{i= - \infty}^\infty (S^\star)^i H_1 \big{\}}^{dd}$, and let $L$ be the disjoint complement of $L_1$ in $X^\star$. Then $\lambda \not \in \sigma(T^\star |L)$ and the set $\lambda \Gamma \cap \sigma(T^\star |L)$ is at most finite. Therefore $L$ is the union of three disjoint $T^\star$-invariant bands $L_2$, $L_3$, and $L_4$ such that $\sigma(T^\star |L_2) \subset \{\alpha \in \mathds{C}: |\alpha| < |\lambda|\}$, $\sigma(T^\star |L_3) \subset \{\alpha \in \mathds{C}: |\alpha| > |\lambda|\}$, and
some power of the operator $T^\star|L_4$ is a multiplication operator and $\lambda \not \in \sigma(T^\star|L_4)$.

To finish the proof it remains to identify $X$ with its canonical image in $X^{\star \star}$ and take $F_i = (L_i)^\star \cap X, i=1, \ldots,4$.
\end{proof}

\begin{corollary} \label{c4} Assume conditions of Theorem~\ref{t10} and assume additionally that the operator $T$ is band irreducible and that the Banach lattice $X$ contains no atoms. Then $\sigma(T,X) = \sigma_{ap}(T,X)$.
\end{corollary}

Similarly to Remark~\ref{r2} it is not clear whether the condition that $X$ has order continuous norm in Theorem~\ref{t10} or Corollary~\ref{c4} can be weakened.

The proof of Theorem~\ref{t8} heavily depends on the fact that for any $x \in \hat{X}$ there is a $y \in X$ such that $|y| \leq |x|$. Of course we cannot expect such a property to be true if we consider analytic subspaces of Banach lattices; nevertheless often we can assume a close property that is described in the following definition.

\begin{definition} \label{d3} Let $X$ be a Banach lattice and $Y$ be a closed subspace of $X$. We say that $Y$ is \textit{almost localized} in $X$ if for any band $Z$ in $X$ and for any $\varepsilon > 0$ there is a $y \in Y$ such that $\|y\|=1$ and $\|(I - P_Z)y\|_{\hat{X}} < \varepsilon$ where $P_Z$ is the band projection on the band $\hat{Z}$ generated by $Z$ in the Dedekind completion $\hat{X}$ of $X$.
\end{definition}

\begin{theorem} \label{t3} Let $X$ be a Dedekind complete Banach lattice and $Y$ be an almost localized subspace of $X$. Let $A$ and $U$ be linear bounded disjointness preserving operators on $X$ such that
\begin{enumerate}
  \item $U \geq 0$ \footnote{The assumption that $U \geq 0$ does not result in loss of generality because any $d$-isomorphism $S$ on a Dedekind complete Banach lattice can be represented as $M|S|$ where $M$ is operator of multiplication on a unimodular function from $C(K)$.} and $\sigma(U) \subseteq \Gamma$.
  \item $A \in Z(X)$.
  \item $AUY \subseteq Y$.
\end{enumerate}

\noindent Let $T = AU$. Then $\sigma_{ap}(T,X) = \sigma_{ap}(T,Y)$.

\end{theorem}

\begin{proof} Let $K$ be the Stonean space of the lattice $X$. It is well known that the center $Z(X)$ of $X$ can be identified with $C(K)$. By Proposition 8.4 in~\cite{AK} the map $ f \to UfU^{-1}, f \in C(K)$ is an isomorphism of algebra $C(K)$. Let $\varphi$ be the homeomorphism of $K$ corresponding to that isomorphism. We identify $A$ with a function from $C(K)$ which will be denoted also by $A$ and introduce an auxiliary operator $S$ on $C(K)$ acting by the formula
$$(Sf)(k) = A(k)f(\varphi(k)). $$
We will prove that $\sigma_{ap}(S,C(K)) = \sigma_{ap}(T,X) = \sigma_{ap}(T,Y)$.

    Part 1. Here we will prove that $\sigma_{ap}(S,C(K)) \subseteq \sigma_{ap}(T,Y)$.

    Assume that $0 \in \sigma_{ap}(S)$. Then $A$ is not invertible in $C(K)$. Fix a positive $\varepsilon$ and a clopen subset $E$ of $K$ such that $|A| \leq \varepsilon$ on $E$. Let $B$ be the band in $X$ corresponding to $E$. Because $Y$ is almost localized in $X$ and $C = U^{-1}B$ is a band in $X$ we can find $y \in Y$ such that $\|y\| = 1$ and $\|(I - P_C)y\| \leq \varepsilon$. Then
    $$\|Ty\| \leq \|T(I - P_C)y\| + \|TP_cy\| \leq \varepsilon \|T\| +\|TP_cy\| . $$
    and
    $$ \|TP_C y\| = \|A U P_C y\| = \|A U P_C U^{-1} U y\| = \|A P_B U y\| \leq \varepsilon \|U\|.$$
     Therefore $\|Ty\| \leq (\|T\| + \|U\|)\varepsilon$ and $0 \in \sigma_{ap}(T,Y)$.

     Next assume that $\lambda \in \sigma_{ap}(S)$ and $\lambda \neq 0$. Without loss of generality we can assume that $\lambda = 1$. Notice that for $n \in \mathds{N}$
     $T^n = A_n U^n $ where $A_n = \prod \limits_{i=0}^{n-1} U^iAU^{-i} \in Z(X)$ and that $S^nf(k) = A_n(k)f(\varphi^n(k)), f \in C(K), k \in K$. By Lemma 3.6 in~\cite{Ki} there is a point $k \in K$ such that
    $$|A_n(k)| \geq 1 \; \mathrm{and} \; |A_n (\varphi^{-n}(k)| \leq 1 , \; n \in \mathds{N} . $$
    Consider two possibilities. (a) Point $k$ is not $\varphi$-periodic. Let us fix $\varepsilon$, $0 < \varepsilon < 1$. Recalling that $\sigma(U) \subseteq \Gamma$ we see that there is $n \in \mathds{N}$ such that
    $$(1 - \varepsilon)^n < \varepsilon /(\|U^{-n}\| \|U^{n+1}\|). \eqno(7)$$
     Let $V$ be a clopen neighborhood of $k$ such that the sets $\varphi^j(V), j \in \mathds{Z}, |j| \leq n+1$ are pairwise disjoint and
    $$ |A_n(p)| \geq 1/2 \; \mathrm{and} \; |A_{n+1}(\varphi^{-n}(p))| \leq 2, \; p \in V            \eqno{(8)}$$
    Fix $y \in Y$ such that $\|y\| = 1$ and $\|(I - P_{\varphi^n(V)})y\| < \delta$, where $\delta$ will be chosen later, and consider
    $z = \sum \limits_{j=0}^{2n} (1- \varepsilon)^{|n-j|} T^j y$. Let $s = P_{\varphi^n(V)}y$ and $t = y - s$. Then $z = u + v$ where $u = \sum \limits_{j=0}^{2n} (1 - \varepsilon)^{|n-j|} T^j s$ and $v = \sum \limits_{j=0}^{2n} (1 - \varepsilon)^{|n-j|} T^j t$. Notice that $\|v\| \leq \delta \big{(} \sum \limits_{j=0}^{2n} \|T^j\|\big{)}$. We proceed now to estimate the norm of $u$ from below and the norm of $Tu - u$ from above. First notice that because the elements $T^js, \; j=0, \ldots, 2n$ are pairwise disjoint we have $\|u\| \geq \|T^n s\|$. Next, $T^n s = A_n U^n s$ and therefore $supp(T^n s) \subseteq supp(U^n s) \subseteq \varphi^{-n}(\varphi^n(V) = V$ whence, in view of $(8)$, $\|T^n s \| \geq 1/2 \|U^n s\| \geq  1/2\|s\|/\|U^{-n}\|$ and therefore
    $$\|u\| \geq (1 - \delta)/(2 \|U^{-n}\|). \eqno{(9)}$$
    On the other hand the definition of $u$ and pairwise disjointness of $T^j s, j=0, \ldots , 2n$ provides the following inequality
    $$ \|Tu - u\| \leq \varepsilon \|u\| + (1-\varepsilon)^n \|s\| + \|(1 - \varepsilon)^n \|T^{2n+1}s\|.\eqno(10)$$
    Notice that (again in view of $(8)$) $\|T^{2n+1}s\| = \|T^{n+1}T^ns\| = \|A_{n+1}U^{n+1}T^ns\| \leq 2 \|U^{n+1}\| \|T^ns\| \leq 2 \|U^{n+1}\|\|u\|$.
    Combining the inequalities $(7)$ and $(10)$ we see that $\|Tu - u\| \leq 4 \varepsilon \|u\|$. Next notice that
    $$\|Tz - z\| \leq \|Tu -u\| +\|v\| + \|Tv\| \leq 4 \varepsilon \|u\| + \delta (1 + \|T\|) \sum \limits_{j=0}^{2n} \|T^j\|.\eqno{(11)}$$
    and
    $$ \|z\| \geq \|u\| -\|v\| \geq \|u\| - \delta \sum \limits_{j=0}^{2n} \|T^j\|.\eqno{(12)}$$
    if we choose $\delta$ in such a way that
    $$\delta < \frac{\varepsilon (1 - \delta)} {4 \|U^{-n}\| (1 + \|T\|) \sum \limits_{j=0}^{2n} \|T^j\| }$$
    then the inequalities $(9)$, $(11)$, and $(12)$ provide
    $$\|Tz - z \| \leq 6\varepsilon \|u\|$$
    and
    $$ \|z\| \geq \|u\|/2. $$
   Finally, because $\varepsilon$ can be chosen arbitrary small we see that $1 \in \sigma_{ap}(T,Y)$.

   Now assume that point $k$ is $\varphi$-periodic with the smallest period $p$. Then by Frolik's theorem~\cite[Theorem 6.25, page 150]{Wa} there is a clopen neighborhood $V$ of $k$ which consists of $\varphi$-periodic points with period $p$. In this case $A_p(k) = 1$. For any $n \in \mathds{N}$ we consider a clopen neighborhood $V_n$ of $k$ and $y_n \in Y$ with properties
   \begin{itemize}
     \item $V_n \subseteq V$.
     \item $\varphi^i(V_n) \cap \varphi^j(V_n) = \emptyset , \; 0 \leq i < j \leq p-1$.
     \item $|A_p(t) - 1| \leq 1/n, \; t \in V_n$.
     \item $\|y_n\| =1$.
     \item $\|y_n - P_n y_n \| \leq 1/n$ where $P_n$ is the band projection corresponding to the set $V_n$

   \end{itemize}
   Let $z_n = \sum \limits_{j=0}^{p-1} T^j y_n$. Then $z_n = u_n + v_n$ where $u_n = \sum \limits_{j=0}^{p-1} P_n T^j y_n$. We can easily see that
    $\|v_n\| \leq \frac{1}{n} \sum \limits_{j=0}^{p-1} \|T^j\|$ and that $\|u_n\| \geq \|P_n y_n\| \geq 1 - \frac{1}{n}$. Therefore
    $\liminf \limits_{n \to \infty} \|z_n\| \geq 1$. On the other hand $Tz_n - z_n = (Tu_n - u_n) + (Tv_n - v_n)$ and we already know that
    $Tv_n - v_n \mathop \rightarrow \limits_{n \to \infty} 0$. Consider $Tu_n - u_n = T^p P_n y_n - P_n y_n = A_p U^p P_n y_n - P_n y_n $. We claim that
    $U^p P_n y_n = P_n y_n$. Indeed, if $f \in C(K)$ then $U^p P_nf U^{-p} = P_nf \circ \varphi^p = P_nf$ whence $U^p P_nf = P_n f U^p$. The center $Z(P_nX)$ of the band $P_n X$ can be identified with $P_n C(K)$ and therefore the operator $U^p P_n = P_n U^p P_n$ is a central operator on $P_nX$. Because this operator is also positive and its spectrum in $P_nX$ lies on the unit circle we see that $U^p P_n = P_n$. Therefore
    $Tu_n - u_n = A_p P_n y_n - P_n y_n  \mathop \rightarrow \limits_{n \to \infty} 0$ whence $1 \in \sigma_{ap}(T,Y)$.

   Part 2. We will prove here that $\sigma_{ap}(T,X) \subseteq \sigma_{ap}(S,C(K))$. Assume that $\lambda \notin \sigma_{ap}(S,C(K))$. First consider $\lambda = 0$. Then clearly the multiplication operator $A$ is invertible as an operator on $C(K)$ whence $T$ is invertible on $X$. Thus we can assume that $\lambda = 1$.
By~\cite[Theorem 3.29]{Ki} and the cited above Frolik theorem the fact that $1 \not \in \sigma_{ap}(S, C(K))$ implies that  $K$ is the union of disjoint $\varphi$-invariant sets $K_1, K_2, K_3$ and $O$ with the properties
\begin{enumerate}[I]
  \item The sets $K_1$ and $K_2$ are closed in $K$, $K_3$ is clopen, and $O$ is open.
  \item There is $p \in \mathds{N}$ such that $\varphi^p(k) = k, k \in K_3$.
  \item $ 1 \notin \sigma(S,C(K_3))$.
  \item $\sigma(S, C(K_1)) \subseteq \{\lambda \in \mathds{C} : \; |\lambda| < 1\}$.
  \item $\sigma(S, C(K_2)) \subseteq \{\lambda \in \mathds{C} : \; |\lambda| > 1\}$.
  \item For any clopen subset $E$ of $O$ and for any clopen neighborhoods $V_1$, $V_2$ of $K_1$ and $K_2$, respectively there is an $m \in \mathds{N}$ such that $\varphi^m(E) \subseteq V_1$ and $\varphi^{-m}(E) \subseteq V_2$.
\end{enumerate}
Assume contrary to our claim that there are $x_n \in X$ such that $\|x_n\| = 1$ and $Tx_n - x_n \mathop \to \limits_{n \to \infty} 0$. It follows immediately from $(II)$, $(III)$ and $U \geq 0$ that $P_{K_3}x_n \to 0$. It also follows from $(VI)$ that there is a clopen neighborhood $V$ of $K_1$ such that $\varphi(V) \subseteq V$.
Then $S$ acts on $C(V)$ and by~\cite[Theorem 3.23]{Ki} $\rho(T,C(V)) = \rho(T, C(K_1)) < 1$ where $\rho(T,X)$ is the spectral radius of $T$ on $X$. Therefore there are $m \in \mathds{N}$ and $a \in (0,1)$ such that $\|A_p\|_{C(V)} \|U^p\| \leq a^p$ for $p \in \{m, m+1, \ldots \}$. Let us fix such $p$. Then $T^p x_n - x_n \to 0$ whence $P_V T^p x_n - P_V x_n \to 0$. But $\|P_V T^p x_n\| = \|P_V A_p U^p x_n\| \leq a^p$ whence $\limsup \|P_V x_n\| \leq a^p$. Because $p$ can be chosen arbitrary large we see that $\limsup \|P_V x_n\| =0$. Therefore we can assume that $supp \, x_n \subseteq K_2 \cup (O \setminus V)$. But then it is not difficult to see from $(V)$ and $(VI)$ that there are $A > 1$ and $m \in \mathds{N}$ such that $\|T^m x_n\| \geq A, \forall \; n \in \mathds{N}$ in obvious contradiction with $T^m x_n - x_n \mathop \to \limits_{n \to 0} 0$.

Part 3. To finish the proof it remains to notice that the inclusion $\sigma_{ap}(T,Y) \subseteq \sigma_{ap}(T,X)$ is trivial.
\end{proof}

Let now $X$ be an arbitrary Banach lattice. If $U$ is a $d$-isomorphism on $X$ and $A \in Z(X)$ then $A$ and $U$ are order continuous operators on $X$ and therefore by a result of Veksler~\cite{Ve} they can be extended in the unique way to order continuous operators $\hat{A}$ and $\hat{U}$ on the Dedekind completion $\hat{X}$ of $X$. Moreover, $\hat{A} \in Z(\hat{X})$ and $\hat{U}$ is a $d$-isomorphism of $\hat{X}$. It is not difficult to show using the order continuity of $U$ and the definition of norm in $\hat{X}$ that if $\sigma (U,X) \subseteq \Gamma$ then $\sigma (\hat{U}, \hat{X}) \subseteq \Gamma$. Notice also that $Y$ is an almost localized subspace of $X$ if and only if it is almost localized in $\hat{X}$. Let $K$ and $\hat{K}$ be the Gelfand compacts of $Z(X)$ and $Z(\hat{X}$, respectively. As above the maps $f \rightarrow UfU^{-1}, \; f \in Z(X)$ and $g \rightarrow \hat{U}g\hat{U}^{-1}, g \in Z(\hat{X})$ induce homeomorphisms $\varphi$ of $K$ and $\hat{\varphi}$ of $\hat{K}$. We introduce the operators $S$ on $C(K)$ and $\hat{S}$ on $C(\hat{K})$ as
$$Sf = A(f \circ \varphi), f \in C(K)$$
and
$$\hat{S}f = A(f \circ \hat{\varphi}), f \in C(\hat{K}).$$
Surely $X$ is almost localized in $\hat{X}$ and therefore Theorem~\ref{t3} and its proof provide the following corollary.

\begin{corollary} \label{c3} Let $X$ be a Banach lattice, $A \in Z(X)$, $U$ be a positive $d$-isomorphism of $X$, and $\sigma(U) \subseteq \Gamma$.

Then $ \sigma_{ap}(T,X) = \sigma_{ap} (\hat{T}, \hat{X})= \sigma_{ap}(S,C(K)) = \sigma_{ap}(\hat{S}, C(\hat{K}))$.

\end{corollary}

But we can claim more.

\begin{theorem} \label{t6} Assume conditions of Corollary~\ref{c3}. Then
$$ \sigma(T,X) = \sigma(\hat{T}, \hat{X})= \sigma(S,C(K)) = \sigma(\hat{S}, C(\hat{K}))$$.

\end{theorem}

\begin{proof} The proof will be divided into three steps.

Step 1. $\sigma(S,C(K)) = \sigma(\hat{S}, C(\hat{K}))$. First of all notice that the map $f \rightarrow \hat{f}, f \in Z(X)$ where $\hat{f}$ is the unique extension of $f$ on $\hat{X}$  is an isomorphic and isometric embedding of $Z(X)$ into $Z(\hat{X})$. To this embedding corresponds the continuous map
$\psi : \hat{K} \rightarrow K$. From the identity $\widehat{UfU^{-1}} = \hat{U}\hat{f}\hat{U}^{-1}$ we obtain that
$$ \varphi(\psi(k)) = \psi(\hat{\varphi}(k), \; k \in \hat{K}. \eqno{(13)}$$
Assume now that $\lambda \in \sigma(\hat{S},C(\hat{K})) \setminus \sigma(S,C(K))$. Then by Corollary~\ref{c3} we have $\lambda \in \sigma_r(\hat{S}, C(\hat{K}))$.
Clearly $\lambda \neq 0$ and it follows from Theorem 3.31 in~\cite{Ki} that there are a point $k \in \hat{K}$, $N \in \mathds{N}$, and a positive $\varepsilon$ such that
$$ |\hat{A}_n(k) \leq (|\lambda| - \varepsilon)^n \; \mathrm{and} \; |\hat{A}(\hat{\varphi}^{-n}(k)) \geq (|\lambda| + \varepsilon)^n, \; n \geq N. \eqno(14) $$
Let $t = \psi(k)$ then it follows from $(13)$ that
$$ |A_n(t)| \leq (|\lambda| - \varepsilon)^n \; \mathrm{and} \; |A(\varphi^{-n}(t))| \geq (|\lambda| + \varepsilon)^n, \; n \geq N. \eqno{(15)} $$
On the other hand it follows from the fact that $\lambda \not \in \sigma(S,C(K))$ and Theorems 3.11 and 3.12 in ~\cite{Ki} that $K$ is the union of three disjoint $\varphi$-invariant sets $K_1$, $K_2$, and $O$ such that
\begin{enumerate}[(i)]
  \item The sets $K_1$ and $K_2$ are closed subsets of $K$, $\sigma(S,C(K_1)) \subseteq \{\gamma \in \mathds{C}: \; |\gamma|< |\lambda|\}$, and
  $\sigma(S,C(K_2)) \subseteq \{\gamma \in \mathds{C}: \; |\gamma|> |\lambda|\}$.
  \item The set $O$ is open in $K$, there is $p \in \mathds{N}$ such that $\varphi^p(s) = s, \; s \in O$, and
  $\sigma(S, cl(O)) \subseteq \{\gamma \in \mathds{C}: \; |\lambda| - \varepsilon/2 < |\gamma| <  |\lambda| + \varepsilon/2 \}$.
  \end{enumerate}
  But the properties $(i)$ and $(ii)$ above clearly contradict $(15)$ and thus we have proved that $\sigma(S,C(K)) \subseteq \sigma(\hat{S},C(\hat{K}))$.

  The inclusion $\sigma(\hat{S},C(\hat{K})) \subseteq \sigma(S,C(K))$ can be proved in a similar way. Indeed, if $\lambda \in \sigma(S,C(K)) \setminus \sigma(\hat{S},C(\hat{K}))$ then $\lambda \in \sigma_r(S,C(K))$ whence there are a point $t \in K$ and a positive $\varepsilon$ satisfying $(15)$. Let
  $k \in \psi^{-1}(t)$ . Then in view of $(13)$ we see that $k$ and $\varepsilon$ satisfy $(14)$ in contradiction with our assumption that
  $\lambda \not \in \sigma(\hat{S},C(\hat{K}))$.

  Step 2. $\sigma(\hat{T}, \hat{X}) = \sigma(\hat{S}, C(\hat{K}))$. Assume that $\lambda \not \in \sigma(\hat{S}, C(\hat{K}))$. Then Theorems 3.11 and 3.12 in ~\cite{Ki} and the Frolik's theorem guarantee that $\hat{K}$ is the union of three disjoint, clopen, $\varphi$-invariant subsets $\hat{K}_1$, $\hat{K}_2$, and $\hat{O}$ (some of them might be empty) such that

 \noindent $(\alpha)$ $\sigma(\hat{S},C(\hat{K}_1)) \subset \{\gamma \in \mathds{C}: \; |\gamma| < |\lambda|\}$.

 \noindent $(\beta)$ $\sigma(\hat{S},C(\hat{K}_2)) \subset \{\gamma \in \mathds{C}: \; |\gamma| > |\lambda|\}$.

 \noindent $(\gamma)$ $\exists p \in \mathds{N}$ such that $\hat{\varphi}^p(k) =k, k \in \hat{O}$, and $\lambda \not \in \sigma(\hat{S},C(\hat{O}))$.

 \noindent To the partition $(\hat{K}_1, \hat{K}_2, \hat{O})$ of the Stonean space $\hat{K}$ corresponds the partition of $\hat{X}$ into three pairwise disjoint bands $\hat{X}_1$, $\hat{X}_2$, and $\hat{X}_O$. These bands are $\hat{U}$-invariant because the sets  $(\hat{K}_1, \hat{K}_2$, $\hat{O})$ are
 $\hat{\varphi}$-invariant; and because $\hat{A}$ is a band preserving operator they are $\hat{T}$-invariant as well. Consider the operator $\hat{T}_1 = \hat{T}|\hat{X}_1$. Then
 $\|\hat{T}_1^n\| \leq \|\hat{A}_n | \hat{X}_1\| \|\hat{U}^n\|$ whence, in view of $(\alpha)$, $\rho(\hat{T}_1) \leq \rho(\hat{S}, C(\hat{K}_1)) \rho(\hat{U}) < |\lambda|$. Next consider the operator $\hat{T}_2 = \hat{T}| \hat{X}_2|$. The condition $(\beta)$ above guarantees that $\hat{T}_2$ is invertible and that $\rho(\hat{T}_2^{-1}) < |\lambda|$ whence
 $\sigma(\hat{T}_2,\hat{X}_2) \subset \{\gamma \in \mathds{C}: |\gamma| > |\lambda|\}$. Finally let us consider the operator $\hat{T}_O = \hat{T}|\hat{X}_O$. If we prove that
 $\lambda \not \in \sigma(\hat{T}_O,\hat{X}_O)$ then it will follow that $\lambda \not \in \sigma(\hat{T}, \hat{X})$. We assume that $\hat{O}\neq \emptyset$ as otherwise it is nothing to prove. By Frolik's theorem there are positive integers $p_1, p_2, \ldots , p_k$ such that $p_i \leq p, i=1, \ldots, k$ such that $\hat{O}$ is the union of nonempty, clopen pairwise disjoint sets $\hat{O}_1, \ldots , \hat{O}_k$ and the set $\hat{O}_i$ consists of $\hat{\varphi}$-periodic points of the smallest period $p_i$. Let $Y_1, \ldots , Y_i$ be the corresponding bands in $\hat{X}_O$. Then $\hat{T}^{p_i}|Y_i = \hat{A}_{p_i}|Y_i$ and it is immediate to see that
 $$\sigma(\hat{T},Y_i) =\sigma(\hat{S},C(\hat{O}_i)) = \{\gamma \in \mathds{C}: \exists t \in \hat{O}_i \; \mathrm{such} \; \mathrm{that} \gamma^{p_i} = \hat{A}_{p_i}(t)\}.\eqno{(16)}$$
In view of $(16)$ and the condition $(\gamma)$ we have $\lambda \not \in \sigma(\hat{T}_O, \hat{X_O})$ whence $\lambda \not \in \sigma(\hat{T}, \hat{X})$ and thus we have proved that $\sigma(\hat{T}, \hat{X}) \subseteq \sigma(\hat{S}, C(\hat{K}))$.

Now assume that $\lambda \not \in \sigma(\hat{T}, \hat{X})$. Then it follows from Theorem 12.20(1) in~\cite{AAK} that there is a band (maybe trivial) $Y$ in $\hat{X}$ such that $\hat{U}Y = Y$, $\exists p \in \mathds{N}$ such that $\hat{\varphi}^p(t) = t, \; t \in supp \; Y$, and $\sigma(\hat{T}, Y^d) \cap \lambda \Gamma = \emptyset$. Then it follows from the proof of Theorem 13.9 in~\cite{AAK} and the fact that $\hat{\varphi}$ is a homeomorphism of $\hat{K}$ that
$Y^d$ is the direct sum of two disjoint (maybe trivial) $T$-invariant bands $\hat{X}_1$ and $\hat{X}_2$ such that $\sigma(\hat{T}|\hat{X}) \subset \{\gamma \in \mathds{C}: |\gamma| < |\lambda|\}$ and $\sigma(\hat{T}|,\hat{X}_2) \subset \{\gamma \in \mathds{C}: |\gamma| > |\lambda|\}$. Applying the same kind of reasoning as above in the proof of the inclusion $\sigma(\hat{T}, \hat{X}) \subseteq \sigma(\hat{S}, C(\hat{K}))$ and based on the fact that
$\sigma(\hat{U}) \subseteq \Gamma$ we see that $\lambda \not \in \sigma(\hat{S}, C(\hat{K}))$.

Step 3. $\sigma(T,X) = \sigma(\hat{T}, \hat{X})$. The inclusion $\sigma(\hat{T}, \hat{X}) \subseteq \sigma(T,X)$ can be proved as in Theorem~\ref{t9}. Indeed, in the proof of Theorem~\ref{t9} we did not use directly that $\hat{X}$ has the weak Fatou property but only the statement of Theorem~\ref{t8}. Here we can instead refer to Theorem~\ref{t3}.

Assume that $\lambda \in \sigma(T,X) \setminus \sigma(\hat{T}, \hat{X})$. We can assume that $\lambda = 1$. Moreover by switching to factor operators like in the proof of Theorem~\ref{t9} we can assume that $\Gamma \subset \sigma_r(T,X)$ and $\Gamma \cap \sigma(\hat{T}, \hat{X}) = \emptyset$. But then
$\hat{X}$ is the direct sum of two disjoint $\hat{T}$-invariant bands $\hat{X}_1$ and $\hat{X}_2$ such that $\sigma(\hat{T}, \hat{X}_1) \subseteq \{\alpha \in \mathds{C} : |\alpha| < 1\}$ and $\sigma(\hat{T}, \hat{X}_2) \subseteq \{\alpha \in \mathds{C} : |\alpha| > 1\}$. Let $\hat{K}_1, \hat{K}_2$  be the corresponding partition of $\hat{K}$. It follows from $\sigma(\hat{U}) \subseteq \Gamma$ that for some large enough $n \in \mathds{N}$ we have
$|\hat{A}_n| < 1$ on $\hat{K}_1$ and $|\hat{A}_n| > 1$ on $\hat{K}_2$. Because $A_n \in Z(X)$ we conclude that $X$ is the direct sum of two disjoint $T$-invariant bands $X_1$ and $X_2$ such that $\sigma(T, X_1) \subseteq \{\alpha \in \mathds{C} : |\alpha| < 1\}$ and $\sigma (T, X_2) \subseteq \{\alpha \in \mathds{C} : |\alpha| > 1\}$ in contradiction with our assumption that $\Gamma \subset \sigma(T,X)$.
\end{proof}

We return now to almost localized subspaces of Banach lattices.

\begin{theorem} \label{t11} Assume conditions of Theorem~\ref{t3}. Then $\sigma(T,X) \subseteq \sigma(T,Y)$.

\end{theorem}

\begin{proof} By Theorem~\ref{t6} we can assume that $X$ is Dedekind complete. Assume that $\lambda \in \sigma(T,X) \setminus \sigma(T,Y)$. Then
by Theorem~\ref{t3} we have $\lambda \in \sigma_r(T,X)$. We can assume that $\lambda = 1$. Combining Theorems~\ref{t6} and~\ref{t3} with Theorems 3.29 and 3.31 from~\cite{Ki} we see that there is a band $B$ in $X$ such that its Banach dual $B^\star$ has the following properties.

\noindent The bands $(S^\star)^{-i}B (S^\star)^i , i \in \mathds{Z}$, are pairwise disjoint.

\noindent For any $F_0 \in B^\star$ there are a sequence $\{F_i \in X^\star, i \in \mathds{Z} \}$, and $\varepsilon > 0$ such that $T^\star F_i = F_{i-1}, i \in \mathds{Z}$, and the vector function $G(\alpha) = \sum \limits_{i=-\infty}^\infty \alpha^i F_i$ is analytic in the annulus $\{1 - \varepsilon < |\alpha| < 1 + \varepsilon \}$.

  Let $y \in B$ be such that $\|y\|=1$ and $\|P_By\|> 3/4$. Then there is an $F\in B^\star$ such that $\langle y,F \rangle \geq 1/2$. On the other hand
  $\langle y, G(\alpha) \rangle \equiv 0 $ in some open neighborhood of $\Gamma$ whence $\langle y,F \rangle =0$, a contradiction.
\end{proof}

\begin{corollary} \label{c5} Assume conditions of Theorem~\ref{t3}. Assume additionally that $Y$ is an analytic subspace of $X$ and that $U^n \neq I$ for any $n \in \mathds{N}$. Then $\sigma(T,Y)$ is a connected rotation invariant subset of $\mathds{C}$.
\end{corollary}

\begin{proof} In virtue of previous results we can assume without loss of generality that $X$ is Dedekind complete. Let us show first that $\sigma(T,Y)$ is rotation invariant. Indeed, otherwise there are a nonzero band $B$ in $X$ and a natural $n$ such that $(T^n -A_n)B =0$. Because $Y$ is an analytic subspace of $X$ it follows that $(T^n - A_n)Y =0$. But $Y$ is almost localized in $X$ whence $\varphi^n(k) = k, \; k \in K$. Therefore $U^nf = fU^n, \; f \in Z(X)$ and thus $U^n \in Z(X)$. But $U \geq 0$ and $\sigma(U) \subseteq \Gamma$ whence $U^n = I$ in contradiction with our assumption.

It remains to prove that the set $\{ |\lambda|: \; \lambda \in \sigma(T,Y)\}$ is connected. Assume to the contrary that there is a positive $r$ such that $Y$ is the direct sum of two nontrivial spectral subspaces $Y_1$ and $Y_2$ , $\sigma(T,Y_1) \subset \{\lambda \in \mathds{C}: \; |\lambda| < r\}$,
and $\sigma(T,Y_2) \subset \{\lambda \in \mathds{C}: \; |\lambda| > r\}$. It follows from Theorems~\ref{t3} and~\ref{t11}, from the fact that $\sigma_r(T,Y)$ is open in $\mathds{C}$, and from Theorem 13.1 in~\cite{AAK}  that $X$  is the direct sum of two $T$-invariant bands $X_1$ and $X_2$ such that $\sigma(T,X_1) \subset \{\lambda \in \mathds{C}: \; |\lambda| < r\}$ and $\sigma(T,X_2) \subset \{\lambda \in \mathds{C}: \; |\lambda| > r\}$. Then obviously $Y_i \subseteq X_i, i=1,2$ in contradiction with $Y$ being an analytic subspace of $X$.
\end{proof}

Let us consider some examples.

\begin{example} \label{e3} Let $\Gamma^n$ be the $n$-torus and $m$ be the normalized Lebesgue measure on $\Gamma^n$. Let $X$ be a Banach space of $m$-measurable functions on $\Gamma^n$ such that $L^\infty (\Gamma^n, m) \subseteq X \subseteq L^1 (\Gamma^n, m)$ and the norm in $X$ is rotation invariant. Let $Y$ be the (closed) subspace of $X$ consisting of all functions from $X$ analytic in the open polydisk $\mathds{D}^n$. Finally, let $U$ be an operator on $Y$ induced by a non periodic rotation on $\Gamma^n$, $A \in H^\infty(\Gamma^n, m)$, and $T = AU$.

Then $\sigma(T,Y)$ is a connected rotation invariant subset of $\mathds{C}$.

Moreover, if $U$ is generated by a strictly ergodic rotation of $\Gamma^n$ then $\sigma(T,Y)$ is a circle centered at $0$ if $A$ is invertible in $H^\infty$ and a disk (or point 0) otherwise.

\end{example}

\begin{example} \label{e4} This example is similar to Example~\ref{e3} but instead of polydisk we consider the unit ball $B^n$ in $\mathds{C}^n$, norm in $X$ is invariant under unitary transformations of $B^n$ and operator $U$ is composition operator generated by a non periodic unitary transformation of $B^n$.

\end{example}

\begin{example} \label{e5} Let $m$ be the normalized Lebesgue measure on $\Gamma$ and $p \in (1, \infty)$. Let $X$ be a rearrangement invariant Banach space of $m$-measurable functions such that its lower and upper Boyd indices (see e.g.~\cite[Section 2(b)]{LT}) are both equal to $p$. An example of such a space is the Lorentz space $\Lambda^p$. Let $\varphi$ be a Mobius transformation of the unit disk $D$. Let $V$ be the composition operator on $X$ generated by $\varphi$, $A$ be an invertible function from the disk algebra $A(D)$, and $T = AU$. Let $Y$ be the closed subspace of $X$ which consists of all functions from $X$ that are boundary values of functions analytic in $D$. Then $T$ is bounded on $X$ and
\begin{enumerate}[(a)]
  \item If $\varphi$ is a non periodic elliptic transformation then $\sigma(T,Y) = A(\xi) \Gamma$ where $\xi$ is the fixed point of $\varphi$ in $D$.
  \item If $\varphi$ is a parabolic transformation then $\sigma(T,Y) = A(\xi) \Gamma$ where $\xi$ is the fixed point of $\varphi$ on $\Gamma$.
  \item If $\varphi$ is a hyperbolic transformation then $\sigma(T,Y)$ is the annulus with radii $|A(\xi_1)|\varphi^\prime (\xi_1)|^{-1/p}$ and
  $|A(\xi_2)|\varphi^\prime (\xi_2)|^{-1/p}$ where $\xi_1$ and $\xi_2$ are fixed points of $\varphi$ on $\Gamma$.
\end{enumerate}
\end{example}

\begin{proof} Let $U = |\varphi^\prime|^{1/p} V$. It is enough to prove that $\sigma(U,X) \subseteq \Gamma$. Indeed, after it proved we can apply our previous results and the formula for spectral radius of weighted composition operators in $C(K)$ (see e.g.~\cite[Theorem 3.23]{Ki}). By Boyd's theorem (see~\cite{LT} or the original Boyd's paper~\cite{Bo}) for any $\varepsilon$, $0 < \varepsilon < p - 1$, $X$ is an interpolation space between $L^{p - \varepsilon}$ and $L^{p + \varepsilon}$ and the constant of interpolation does not depend on $\varepsilon$. The operator $U$ is bounded on $L^{p -\varepsilon}$ and on $L^{p + \varepsilon}$ whence it is bounded on $X$. Next notice that $U = |\varphi^\prime|^{\frac{-\varepsilon}{p(p - \varepsilon)}}W_{1,\varepsilon}$ and $U = |\varphi^\prime|^{\frac{\varepsilon}{p(p + \varepsilon)}}W_{2,\varepsilon}$ where $W_{1, \varepsilon}$ and
$W_{1, \varepsilon}$ are isometries in $L^{p - \varepsilon}$ and $L^{p + \varepsilon}$, respectively. Applying Theorem~\ref{t3} and Boyd's interpolation theorem we see that $\sigma(U,X) \subseteq \Gamma$.
\end{proof}

Let us return to the general case when we can assume the conditions of Theorem~\ref{t3}. Simple examples show that in general $\sigma(T,X) \subsetneqq \sigma(T,Y)$. Therefore it is natural to ask what can be said about the set $\sigma(T,Y) \setminus \sigma(T,X)$. Assume in addition to the conditions of Corollary~\ref{c3} that $UY = Y$. Then $AY \subseteq Y$. Let $Z(X)$ be the ideal center of $X$.
 Let $\mathcal{A} = \{f \in Z(X) : \; fY \subseteq Y \}$. Then $\mathcal{A}$ is a unital closed subalgebra of $Z(X)$. We will denote its space of maximal ideals by $\mathfrak{M}$ and its Shilov boundary by $\partial$. We will identify an element $a \in \mathcal{A}$ and its Gelfand image. Notice that the map $a \rightarrow U^{-1}aU$ is an automorphism of $\mathcal{A}$. Let $\phi$ be the corresponding homeomorphism of $\mathfrak{M}$. We can consider weighted composition operator $Sf = a(f \circ \phi)$ either on $C(\mathfrak{M})$ or on $C(\partial)$.

 \begin{lemma}\label{l3} Let $a \in \mathcal{A}$

 $(1)$ $\|a\|_{Z(X)} = \|a\|_{\mathcal{L}(Y)}$ and

 $(2)$ $a$ is invertible in $\mathcal{A}$ if and only is it is invertible
 in $\mathcal{L}(Y)$.
 \end{lemma}

 \begin{proof} $(1)$ Clearly $\|a\|_{\mathcal{L}(Y)} \leq \|a\|_{\mathcal{L}(X)} = \|a\|_{Z(X)}$. To prove the converse inequality let us assume that
 $ \|a\|_{Z(X)} = 1$ and take $\varepsilon > 0$. Let $K$ be the Stonean compact of $\hat{X}$. Recall that $Z(X)$ is isometrically embedded into $C(K)$. Let $E = cl \{k \in K : \; |a(k)| > 1 - \varepsilon \}$ and let $B$ be the band in $\hat{X}$ corresponding to $E$. Because $Y$ is almost localized in $X$ there is $y \in Y$ such that $\|y\|=1$ and $\|(I - P_B)y\|_{\hat{X}} < \varepsilon$. Then $\|ay\| \geq (1 - \varepsilon)^2$ whence $\|a\|_{\mathcal{L}(Y)} = 1$.

 $(2)$. Clearly, if $a$ is invertible in $\mathcal{A}$ then it is invertible in $\mathcal{L}(Y)$. Assume that $a$ is invertible in $\mathcal{L}(Y)$.
 It follows easily from the fact that $Y$ is almost localized in $X$ that $a$ is invertible in $Z(X)$ (indeed, otherwise we can for any $\varepsilon > 0$ find a $y \in Y$ such that $\|y\| = 1$ but $\|ay\| < \varepsilon$). Let $b$ be the inverse of $a$ in $\mathcal{L}(Y)$ and $c$ its inverse in $Z(X)$. Then for any $y \in Y$ we have $(b-c)ay = 0$, but $aY = Y$ whence $cY \subseteq Y$ and $a$ is invertible in $\mathcal{A}$.
 \end{proof}

  \begin{proposition} \label{p1} Assume conditions of Theorem~\ref{t3} and that $UY = Y$. Then

 $|\sigma(T,Y)| = |\sigma (S,C(\mathfrak{M})|$.

 \end{proposition}

 \begin{proof} Assume that $r \not \in |\sigma(T,Y)|$. Notice that $\rho(T,Y) = \rho(S,C(\mathfrak{M}))$. By Lemma~\ref{l3} the operators $T$ and $S$ are either both invertible or both not invertible; and if they are both invertible then $\rho(T^{-1}) = \rho(S^{-1})$. Therefore let us consider the case when $Y$ is the direct sum of nonzero spectral subspaces $Y_1$ and $Y_2$, $\sigma(T,Y_1) \subset \{\lambda \in \mathds{C}: |\lambda| < r\}$, and
$\sigma(T,Y_2) \subset \{\lambda \in \mathds{C}: |\lambda| > r\}$. Then by Theorem~\ref{t11} and by Theorem 13.1 in~\cite{AAK} $X$ is the sum of two disjoint $T$-invariant bands $X_1$ and $X_2$, $\sigma(T,X_1) \subset \{\lambda \in \mathds{C}: |\lambda| < r\}$, and $\sigma(T,X_2) \subset \{\lambda \in \mathds{C}: |\lambda| > r\}$. Let $P$ be the band projection on the band $X_1$. Then $PY = Y_1$ whence $P \in \mathcal{A}$. By the Shilov idempotent theorem $\mathfrak{M}$ is the union of two clopen sets $\mathfrak{M}_1$ and $\mathfrak{M}_2$. It is routine to verify that these sets are $\phi$-invariant and that $\sigma(S,C(\mathfrak{M}_1)) \subset \{\lambda \in \mathds{C}: |\lambda| < r\}$ and $\sigma(S,C(\mathfrak{M}_2)) \subset \{\lambda \in \mathds{C}: |\lambda| > r\}$.

Conversely, if $\sigma(S, C(\mathfrak{M})) = \sigma_1 \cup \sigma_2$ where $\sigma_1   \subset \{\lambda \in \mathds{C}: |\lambda| < r\}$
and $\sigma_2   \subset \{\lambda \in \mathds{C}: |\lambda| > r\}$ then (see~\cite{Ki}) $\mathfrak{M}$ is the union of two clopen $\phi$-invariant sets $\mathfrak{M}_1$ and $\mathfrak{M}_2$ and $\sigma(S, C(\mathfrak{M}_i)) = \sigma_i, i = 1,2$. Therefore the spectral projections $P_1$ and $P_2= I - P_1$ corresponding to $\sigma_1$ and $\sigma_2$ belong to $\mathcal{A}$. Let $Y_i = P_iY, i = 1,2$, then it is immediate to see that $TY_i \subseteq Y_i, i=1,2$ and that $|\sigma(T|Y_1)| \subset (0,r)$, $|\sigma(T|Y_2)| \subset (r,\infty)$.

 \end{proof}

We will just outline the proof of the following proposition.

\begin{proposition} \label{p2} Assume conditions of Proposition~\ref{p1}. Then $\sigma(S, C(\mathfrak{M})) \subseteq \sigma(T,Y)$.
\end{proposition}

\textit{Sketch of the proof}. We have to consider three possibilities.

\noindent $(1)$ $\lambda \in \sigma_{ap}(S, C(\mathfrak{M}))$. Applying part (1) of Theorem 4.2 in~\cite{Ki}, Lemma 3.6 from~\cite{Ki} and Theorem~\ref{t6} we conclude that $\lambda \in \sigma_{ap}(S, C(\mathfrak{M}))$.

 \noindent $(2)$ $\lambda \in \sigma_r(S, C(\mathfrak{M}))\cap \sigma_r(C(\partial))$. We apply Theorem 3.26 from~\cite{Ki} and Theorem~\ref{t6} to conclude that $\lambda \in \sigma_r(T,Y)$.

 \noindent $(3)$ $\lambda \in \sigma_r(S, C(\mathfrak{M})) \setminus \sigma(S, C(\partial))$. In this case if $\lambda \not \in \sigma(T,Y)$ then
 $\lambda \Gamma \cap \sigma(T,Y) = \emptyset$ in contradiction with Proposition~\ref{p1}. $\Box$

Propositions~\ref{p1} and~\ref{p2} beg the question whether $\sigma(T,Y) = \sigma(S, C(\mathfrak{M}))$. Without additional assumptions it might be false. Indeed (see~\cite{Wi} ) there are a Banach lattice $X$ with trivial center and an isometry $U$ of $X$ such that $U^n \neq I, n \in \mathds{N}$. \footnote{ I am grateful to A. W. Wickstead for information about the corresponding example.} Then
$\sigma(U,X) = \Gamma$ but $\sigma(S, C(\mathfrak{M})) = \{1\}$.

Therefore if we want the equality $\sigma(T,Y) = \sigma(S, C(\mathfrak{M}))$ to be true we must assume some "richness" of the algebra $\mathcal{A}$.

\begin{conjecture} \label{co1} Assume conditions of Proposition~\ref{p1} and assume additionally that $\mathcal{A}$ is almost localized in $C(K)$ where $K$ is the Stonean compact of $X$. Then

\noindent $(1)$ $\sigma(T,Y) = \sigma(S, C(\mathfrak{M}))$.

\noindent $(2)$ $\sigma_{ap}(T,Y) = \sigma_{ap}(S, C(\partial))$.

\end{conjecture}

\begin{remark} \label{r3} There are two arguments in favor of Conjecture~\ref{co1} to be correct.

\noindent First, it is true when $Y$ is a unital uniform algebra (see~\cite[Theorem3.26]{Ki}).

\noindent Second, it is not difficult to prove that it is true if the algebra $\mathcal{A}$ is analytic in $C(K)$.

\end{remark}

\end{document}